\documentclass[oneside,english]{amsart}
\usepackage[T1]{fontenc}
\usepackage[latin9]{inputenc}
\usepackage{mathrsfs}
\usepackage{amstext}
\usepackage{amsthm}
\usepackage{amssymb}

\makeatletter
  \theoremstyle{definition}
  \newtheorem{defn}{\protect\definitionname}
  \theoremstyle{plain}
  \newtheorem*{thm*}{\protect\theoremname}
  \theoremstyle{plain}
  \newtheorem{cor}{\protect\corollaryname}
\theoremstyle{plain}
\newtheorem{thm}{\protect\theoremname}
  \theoremstyle{plain}
  \newtheorem{lem}{\protect\lemmaname}

\usepackage{tikz}
\usetikzlibrary{matrix,arrows}
\tikzset{node distance=2cm, auto}

\makeatother

\usepackage{babel}
  \providecommand{\definitionname}{Definition}
  \providecommand{\lemmaname}{Lemma}
  \providecommand{\theoremname}{Theorem}
\providecommand{\corollaryname}{Corollary}
\providecommand{\theoremname}{Theorem}

\begin{document}

\title[Rational function in strongly convergent variables]{Non-commutative rational function in strongly convergent random variables}

\author{Sheng Yin}
\begin{abstract}
Random matrices like GUE, GOE and GSE have been studied for decades
and have been shown that they possess a lot of nice properties. In
2005, a new property of independent GUE random matrices is discovered
by Haagerup and Thorbjørnsen in their paper \cite{HT05}, it is called
strong convergence property and then more random matrices with this
property are followed (see \cite{Sch05}, \cite{CD07}, \cite{And13},
\cite{Mal12}, \cite{CM14} and \cite{BC16}). In general, the definition
can be stated for a sequence of tuples over some $\text{C}^{\ast}$-algebras.
And in this general setting, some stability property under reduced
free product can be achieved (see Skoufranis \cite{Sko15} and Pisier
\cite{Pis16}), as an analogy of the result by Camille Male \cite{Mal12}
for random matrices.

In this paper, we want to show that, for a sequence of strongly convergent
random variables, non-commutative polynomials can be extended to non-commutative
 rational functions under certain assumptions. Roughly speaking, the
strong convergence property is stable under taking the inverse. As
a direct corollary, we can conclude that for a tuple $(X_{1}^{\left(n\right)},\cdots,X_{m}^{\left(n\right)})$
of independent GUE random matrices, $r(X_{1}^{\left(n\right)},\cdots,X_{m}^{\left(n\right)})$
converges in trace and in norm to $r(s_{1},\cdots,s_{m})$ almost
surely, where $r$ is a rational function and $(s_{1},\cdots,s_{m})$
is a tuple of freely independent semi-circular elements which lies
in the domain of $r$.
\end{abstract}

\thanks{The author acknowledges the support from the ERC Advanced Grant \textquotedblleft Non-Commutative
Distributions in Free Probability\textquotedblright{} (Grant No. 339760).\\
The author wants to express his thanks to several people who really
help a lot on this work. First, the author wants to thank his supervisor,
Roland Speicher, for suggesting this interesting project and his continued
support. And the author is grateful to Guillaume Cébron for the really
fruitful discussion and some crucial ideas on this article. Finally,
the author wants to thank Tobias Mai and Felix Leid, from whom the
author learn a lot about rational functions.}
\maketitle

\section{Introduction}

In 1990's, a deep relation between random matrices and free probability
was revealed in the paper \cite{Vol91} by Voiculescu. In this paper,
Voiculescu proved that if $(X_{1}^{\left(n\right)},\cdots,X_{m}^{\left(n\right)})$
is a tuple of independent $n\times n$ normalized Hermitian Gaussian
random matrices for each $n\in\mathbb{N}$, then all the moments converge,
i.e.,
\[
\lim_{n\rightarrow\infty}\mathbb{E}\left\{ \text{tr}_{n}(p(X_{1}^{\left(n\right)},\cdots,X_{m}^{\left(n\right)}))\right\} 
\]
exists for any non-commutative polynomial $p$, where we denote the
normalized trace by $\text{tr}_{n}$. Furthermore, we can realize
the limits as a tuple of freely independent semi-circular elements
$(s_{1},\cdots,s_{m})$ in some $\text{C}^{\ast}$-probability space
$\left(\mathcal{A},\tau\right)$, namely, a unital $\text{C}^{\ast}$-algebra
with a state $\tau$. So we can write
\[
\lim_{n\rightarrow\infty}\mathbb{E}\left\{ \text{tr}_{n}(p(X_{1}^{\left(n\right)},\cdots,X_{m}^{\left(n\right)}))\right\} =\tau(p(s_{1},\cdots,s_{m}))
\]
for any polynomial $p$. This result has been extended to some other
random matrix models, for example, a tuple of Wigner matrices with
some assumptions on moments of entries \cite{Dyk93}. On the other
hand, it is also known that this convergence for random matrices can
be improved to the almost sure convergence, see Hiai, Petz \cite{HP00}
and Thorbjørnsen \cite{Tho00}.

Later, Haagerup and Thorbjørnsen showed that the convergence of random
matrices can happen be in a stronger sense, that is, convergence in
the norm. To be precise, in \cite{HT05}, they showed that for any
polynomial $p$,
\begin{equation}
\lim_{n\rightarrow\infty}\left\Vert p(X_{1}^{\left(n\right)}\left(\omega\right),\cdots,X_{m}^{\left(n\right)}\left(\omega\right))\right\Vert =\left\Vert p(s_{1},\cdots,s_{m})\right\Vert \label{eq: strong convergence}
\end{equation}
for almost every $\omega$ in the underlying probability space. Then
we say that $(X_{1}^{\left(n\right)},\cdots,$ $X_{m}^{\left(n\right)})$
\emph{strongly converges} and $(s_{1},\cdots,s_{m})$ is its \emph{strong
limit}. Following their work, Schultz \cite{Sch05} shows that GOE
and GSE also admit semi-circular elements as strong limit. Then Capitaine
and Donati-Martin \cite{CD07} and Anderson \cite{And13} generalize
the result to certain Wigner matrices. Capitaine and Donati-Martin
\cite{CD07} also extend the result to Wishart matrices with free
Poisson elements as strong limit.

Moreover, in the paper \cite{Mal12} by Male, he finds that a tuple
of random matrices from GUE can be enlarged with another tuple of
random matrices who has a strong limit under certain independence
and freeness assumptions. Later, in the paper \cite{CM14} by Collins
and Male, they show that this property also holds for Haar unitary
matrices. And then in the paper \cite{BC16} by Belinschi and Capitaine,
they proved that this property also holds for certain Wigner matrices.

Meanwhile, in recent papers by Skoufranis \cite{Sko15} and Pisier
\cite{Pis16}, it is shown that the strong convergence property is
preserved when adjoining two tuples of non-commutative random variables
which admit strong limits and are free from each other. In other words,
they proved that the reduced free product is stable with respect to
strong convergence.

Therefore, these results show that the strong convergence property
is stable under some algebraic operations, so it is natural to ask
if the strong convergence is stable under another basic algebraic
operation, namely, taking inverses. And then we can hope that the
polynomials in (\ref{eq: strong convergence}) can be replaced by
rational functions under some assumption.

On the other hand, we know that one of the main ingredients used by
Haagerup and Thorbjørnsen is the so-called linearization trick, see
\cite{HT05,HST06} for the idea and details. Inspired by the fact
that such a linearization also holds for non-commutative rational
expressions or rational functions, we can expect an affirmative answer
to our question. In this paper, we will show that this result is indeed
true but the linearization technique is not essentially necessary
when we are going from polynomials to rational functions.

In the following, we always consider the strong convergence in the
faithful tracial $\text{C}^{\ast}$-probability space setting.
\begin{defn}
Let $(\mathcal{A}^{\left(n\right)},\tau^{\left(n\right)})$, $n\in\mathbb{N}$
and $\left(\mathcal{A},\tau\right)$ be some $\text{C}^{\ast}$-probability
spaces with faithful traces. Then we assume that $x^{\left(n\right)}=(x_{1}^{\left(n\right)},\cdots,x_{m}^{\left(n\right)})$
is a tuple of elements from $\mathcal{A}^{\left(n\right)}$ for each
$n\in\mathbb{N}$, and $x=(x_{1},\cdots,x_{m})$ is a tuple of elements
in $\left(\mathcal{A},\tau\right)$ s.t. $x^{\left(n\right)}$ strongly
converges to $x$. That is, they satisfy the following:
\[
\begin{array}{c}
\lim\limits _{n\rightarrow\infty}\tau^{\left(n\right)}\left(p(x^{\left(n\right)},(x^{\left(n\right)})^{\ast})\right)=\tau(p(x,x^{\ast})),\\
\lim\limits _{n\rightarrow\infty}\left\Vert p(x^{\left(n\right)},(x^{\left(n\right)})^{\ast})\right\Vert _{\mathcal{A}^{\left(n\right)}}=\left\Vert p\left(x,x^{\ast}\right)\right\Vert _{\mathcal{A}}
\end{array}
\]
for any polynomial $p$ in $2m$ non-commuting indeterminates.
\end{defn}
In the second section, we will give a concise introduction to rational
functions and rational expressions and some of their relevant properties.
Then, in the last section, we are going to prove the main theorem:
\begin{thm*}
If $x^{\left(n\right)}=(x_{1}^{\left(n\right)},\cdots,x_{m}^{\left(n\right)})$
strongly converges to $x=(x_{1},\cdots,x_{m})$, then for any rational
expression $r$, $r(x,x^{\ast})$ is the limit of $r(x^{\left(n\right)},(x^{\left(n\right)})^{\ast})$
in trace and in norm, provided that $(x,x^{\ast})$ lies in the domain
of $r$.
\end{thm*}
The basic idea behind this is that from polynomials to rational expressions,
our only obstacle is due to taking the inverse. But we will see that
the inverse can be approximated by polynomials uniformly in all dimensions,
hence we can reduce the convergence of rational expressions to the
result on polynomials and also show that $(x^{\left(n\right)},(x^{\left(n\right)})^{\ast})$
will lie in the domain eventually.

As an example or consequence, we can apply our main result to any
random matrices which have a strong limit.
\begin{cor}
Let $X^{\left(n\right)}=(X_{1}^{\left(n\right)},\cdots,X_{m}^{\left(n\right)})$
be a tuple of independent $n\times n$ random matrices for each $n\in\mathbb{N}$,
and $x=(x,\cdots,x_{m})$ a tuple of freely independent random variables
in some faithful tracial $\text{C}^{\ast}$-probability space $\left(\mathcal{A},\tau\right)$.
Assume that $X^{\left(n\right)}$ strongly converges to $x$ almost
surely. Then for any rational expression $r$ with $(x,x^{\ast})$
in its domain, we have $(X^{\left(n\right)}\left(\omega\right),(X^{\left(n\right)}\left(\omega\right))^{\ast})$
lies in the domain of $r$ eventually and
\[
\begin{array}{c}
\lim\limits _{n\rightarrow\infty}\text{tr}_{n}(r(X^{\left(n\right)}\left(\omega\right),(X^{\left(n\right)}\left(\omega\right))^{\ast}))=\tau(r(x,x^{\ast})),\\
\lim\limits _{n\rightarrow\infty}\left\Vert r(X^{\left(n\right)}\left(\omega\right),(X^{\left(n\right)}\left(\omega\right))^{\ast})\right\Vert =\left\Vert r(x,x^{\ast})\right\Vert _{\mathcal{A}}
\end{array}
\]
for almost every $\omega$ in the underlying space.
\end{cor}
In particular, it allows us to claim that a rational expression in
independent GUE random matrices converges almost surely in trace to
the same rational expression in free independent semi-circular elements.
In fact, such a result is not surprising at all. In the recent paper
\cite{HMS15} by Helton, Mai and Speicher, they extended the method
used for the calculation of the distribution of polynomials in free
random variables to the rational case, based on the fact that linearization
works equally well for rational expressions. From their simulation
in Section 4.7 of \cite{HMS15}, we can expect that a rational expression
in independent Gaussian random matrices should almost surely converge
in distribution to the same rational expression in free semi-circular
elements. By our theorem this is true whenever we have random matrices
which admit strong limits.

\section{Rational functions and their recursive structure}

In this section, we will give a short introduction to rational functions
and rational expressions with some highlights which are necessary
for our result in the next section.

It is well-known that for each integral domain, we can construct the
unique quotient field, namely, the smallest field in which this integral
domain can be embedded. This was generalized to certain non-commutative
rings with a property called the Ore condition. This condition can
allow us to construct the field in essentially the same way as in
the commutative case. However, to extend such embedding results to
more general cases requires new ideas.

For example, the ring of polynomial in any $m$ ($m\geqslant2$) non-commuting
indeterminates doesn't satisfy Ore condition due to its non-commutative
nature. So it is not quite obvious that a field of fractions of non-commutative
polynomials really exists and that such a field is unique even if
it exists.

As some necessary conditions for the embeddability of (non-commutative)
rings into fields are known, there was an effort to find some equivalent
conditions. See the work by Malcev \cite{Mal37}, which gives an example
of nonembeddable ring in fields with some necessary condition, and
by Klein \cite{Kle72}, which gives a conjecture for this embeddability
problem.

From 1960's, Cohn began to study the problem of embedding non-commutative
rings into fields and then he developed a matrix method to introduce
the matrix ideals, as the analogue of the ideals in commutative case.
He showed that the prime matrix ideals can be used to describe some
``kernels'' of the embeddings of rings into skew fields, as every
prime ideal in a commutative ring arises as the kernel of a homomorphism
into some commutative field. And this characterization allows us to
derive a criteria for the embeddability of rings into fields.

In the following, we always use $\mathscr{P}$ to denote the non-commutative
polynomials ring and $\mathscr{R}$ the field of fractions obtained
from $\mathscr{P}$ by Cohn's construction. We won't go into details
of this construction but we will talk about some basic properties
to show what do these rational functions look like. In fact, the only
thing about Cohn's construction we shall need is the following theorem:
\begin{thm}
\label{foundamental structure thm for rational functions}Let $r\in\mathscr{R}$
be a rational function, then there exists some $n\in\mathbb{N}$,
a matrix of polynomials $A\in M_{n}(\mathscr{P})$, a row of polynomials
$u\in M_{1,n}(\mathscr{P})$ and a column of polynomials $v\in M_{n,1}(\mathscr{P})$
s.t. $A$ is invertible in $M_{n}(\mathscr{R})$ and $r=uA^{-1}v$.
\end{thm}
For a more general statement and the proof, see \cite[Ch 7]{Coh85}.

In fact, to represent a rational function in terms of matrices of
polynomials appears not only in the context of ring theory, but also
in the system and control theory, called ``realization''. Moreover,
such a realization is usually required to be in a linear form, i.e.,
all the entries in the matrices in the above theorem are at most of
degree $1$ as polynomials. So this technique is also called linearization.
But we won't talk any more about this realization or linearization
technique in this paper, though it has a variety of implications in
different areas.

Now we want to use this theorem to show that the field of rational
functions has a recursive structure. That is, all the rational functions
can be obtained by taking finitely many algebraic operations (addition,
multiplication, inversion) from polynomials. This exactly meets what
we would expect for rational functions intuitively but may not be
obvious from the theory of Cohn.

Denote $\mathscr{R}_{0}=\mathscr{P}$, and by $\mathscr{R}_{1}$ we
denote the subring of $\mathscr{R}$ generated by $\mathscr{R}_{0}\cup\mathscr{R}_{0}^{-1}$,
where $\mathscr{R}_{0}^{-1}$ is the set of inverses of all nonzero
polynomials. Now, suppose that we have constructed the subring $\mathscr{R}_{n}\subseteq\mathscr{R}$
for some $n\in\mathbb{N}$, then we let $\mathscr{R}_{n+1}$ be the
subring of $\mathscr{R}$ generated by $\mathscr{R}_{n}\cup\mathscr{R}_{n}^{-1}$,
where $\mathscr{R}_{n}^{-1}$ is the set of inverses of all nonzero
rational functions in $\mathscr{R}_{n}$. So we have a increasing
sequence of subrings $\{\mathscr{R}_{n}\}_{n\geqslant1}$ in $\mathscr{R}$.
Then we set
\[
\mathscr{R}_{\infty}=\bigcup_{n=1}^{\infty}\mathscr{R}_{n}.
\]
We expect (and will show below) that have $\mathscr{R}_{\infty}=\mathscr{R}$.
The following argument is based on a similar idea for proving that
$\mathscr{R}$ is really a ``free'' field, i.e., every 0 identity
comes from algebraic manipulations. For a reference, see \cite{CR99}
and also \cite{Coh06}.

First, for reader's convenience, we give a short proof for a well-known
lemma about Schur complements for matrices in a unital algebra setting.
\begin{lem}
Suppose that $\mathcal{A}$ is a complex and unital algebra. Let $k,l\in\mathbb{N}$,
$A\in M_{k}\left(\mathcal{A}\right)$, $B\in M_{k\times l}\left(\mathcal{A}\right)$,
$C\in M_{l\times k}\left(\mathcal{A}\right)$ and $D\in M_{l}\left(\mathcal{A}\right)$
s.t. $D$ is invertible. Then the matrix
\[
\begin{pmatrix}A & B\\
C & D
\end{pmatrix}
\]
 is invertible in $M_{k+l}\left(\mathcal{A}\right)$ iff the Schur
complement $A-BD^{-1}C$ is invertible in $M_{k}\left(\mathcal{A}\right)$.
In this case, we will have
\begin{equation}
\begin{pmatrix}A & B\\
C & D
\end{pmatrix}^{-1}=\begin{pmatrix}1 & 0\\
-D^{-1}C & 1
\end{pmatrix}\begin{pmatrix}(A-BD^{-1}C)^{-1} & 0\\
0 & D^{-1}
\end{pmatrix}\begin{pmatrix}1 & -BD^{-1}\\
0 & 1
\end{pmatrix}.\label{inverse of block matrix}
\end{equation}
\end{lem}
\begin{proof}
It's easy to check that
\[
\begin{pmatrix}A & B\\
C & D
\end{pmatrix}=\begin{pmatrix}1 & BD^{-1}\\
0 & 1
\end{pmatrix}\begin{pmatrix}A-BD^{-1}C & 0\\
0 & D
\end{pmatrix}\begin{pmatrix}1 & 0\\
D^{-1}C & 1
\end{pmatrix}
\]
holds whenever $D$ is invertible. Since the matrices
\[
\begin{pmatrix}1 & BD^{-1}\\
0 & 1
\end{pmatrix}\text{ and }\begin{pmatrix}1 & 0\\
D^{-1}C & 1
\end{pmatrix}
\]
are clearly invertible in $M_{k+l}\left(\mathcal{A}\right)$, the
equivalence of invertibilities of $\begin{pmatrix}A & B\\
C & D
\end{pmatrix}$ and $A-BD^{-1}C$ follows immediately. And (\ref{inverse of block matrix})
follows from a simple calculation
\begin{eqnarray*}
\begin{pmatrix}A & B\\
C & D
\end{pmatrix}^{-1} & = & \begin{pmatrix}1 & 0\\
D^{-1}C & 1
\end{pmatrix}^{-1}\begin{pmatrix}A-BD^{-1}C & 0\\
0 & D
\end{pmatrix}^{-1}\begin{pmatrix}1 & BD^{-1}\\
0 & 1
\end{pmatrix}^{-1}\\
 & = & \begin{pmatrix}1 & 0\\
-D^{-1}C & 1
\end{pmatrix}\begin{pmatrix}(A-BD^{-1}C)^{-1} & 0\\
0 & D^{-1}
\end{pmatrix}\begin{pmatrix}1 & -BD^{-1}\\
0 & 1
\end{pmatrix}.
\end{eqnarray*}
\end{proof}
With the help of the above lemma, we can show the following lemma,
which is crucial for our statement on $\mathscr{R}_{\infty}=\mathscr{R}$.
\begin{lem}
If an $n$-by-$n$ matrix $A\in M_{n}\left(\mathscr{R}_{\infty}\right)$
is invertible in $M_{n}\left(\mathscr{R}\right)$, then $A^{-1}\in M_{n}\left(\mathscr{R}_{\infty}\right)$.
\end{lem}
\begin{proof}
We are going to prove this by induction on the size of matrices. First,
let $r\in M_{1}\left(\mathscr{R}_{\infty}\right)$, then we can view
it as a rational function in $\mathscr{R}_{\infty}$, which implies
that there is some $k\in\mathbb{N}$ s.t. $r\in\mathscr{R}_{k}$.
Thus, $r$ is invertible in $M_{1}\left(\mathscr{R}\right)=\mathscr{R}$
means that $r\neq0$, and so we have $r^{-1}\in\mathscr{R}_{k}^{-1}\subseteq\mathscr{R}_{k+1}\subseteq\mathscr{R}_{\infty}$.

Now assume that the claim is true for matrices of size $n-1$. Let
$A\in M_{n}\left(\mathscr{R}_{\infty}\right)$ be invertible in $M_{n}\left(\mathscr{R}\right)$,
then, WLOG, we can write
\[
A=\begin{pmatrix}B & u\\
v & p
\end{pmatrix}
\]
with $p\neq0$, because we can multiply by a permutation matrix to
achieve this. Hence, we see that $B-up^{-1}v\in M_{n-1}\left(\mathscr{R}_{\infty}\right)$
is invertible in $M_{n-1}\left(\mathscr{R}\right)$ by the previous
lemma, then it follows that $\left(B-up^{-1}v\right)^{-1}\in M_{n-1}\left(\mathscr{R}_{\infty}\right)$
by the induction. Since
\[
A^{-1}=\begin{pmatrix}I_{n-1} & 0\\
-p^{-1}u & 1
\end{pmatrix}\begin{pmatrix}(B-up^{-1}v)^{-1} & 0\\
0 & p^{-1}
\end{pmatrix}\begin{pmatrix}I_{n-1} & -vp^{-1}\\
0 & 1
\end{pmatrix}
\]
by (\ref{inverse of block matrix}), we can see clearly that $A^{-1}\in M_{n}\left(\mathscr{R}_{\infty}\right)$
since each matrix in the right hand side lies in $M_{n}\left(\mathscr{R}_{\infty}\right)$.
This completes the proof.
\end{proof}
\begin{thm}
We have
\[
\mathscr{R}=\mathscr{R}_{\infty}.
\]
\end{thm}
\begin{proof}
Let $r\in\mathscr{R}$ be a rational function, then, by Theorem \ref{foundamental structure thm for rational functions},
there exists a matrix of polynomials $A\in M_{n}\left(\mathscr{P}\right)$,
a row $u\in M_{1,n}\left(\mathscr{P}\right)$ and a column $v\in M_{n,1}\left(\mathscr{P}\right)$
for some $n\in\mathbb{N}$ s.t. $A$ is invertible in $M_{n}\left(\mathscr{R}\right)$
and $r=uA^{-1}v$. By the previous lemma, and since $\mathscr{P}\subseteq\mathscr{R}_{\infty}$,
we see that $A^{-1}\in M_{n}\left(\mathscr{R}_{\infty}\right)$ and
thus $r\in\mathscr{R}_{\infty}$.
\end{proof}
It is well-known that in the commutative case, every rational function
can be written in a form like $pq^{-1}$, where $p$ and $q$ are
polynomials. This means that we will have $\mathscr{R}=\mathscr{R}_{1}=\mathscr{R}_{n}$
for all $n\geqslant1$. But it is not true any more for non-commutative
rational functions due to its noncommutativity. For example, we can't
write $xy^{-1}x\in\mathscr{R}_{1}$ as the product $pq^{-1}$ with
two polynomials $p,q$. And the rational function $\left(x^{-1}+y^{-1}+z^{-1}\right)^{-1}$
lies in $\mathscr{R}_{2}$ but not in $\mathscr{R}_{1}$.

On the other hand, we should note that such a representation is not
unique. For a simple example,
\[
r\left(x,y\right)=\left(xy\right)^{-1}=y^{-1}x^{-1}\in\mathscr{R}_{1},
\]
we can see that we can use one polynomial $xy$ or two polynomials
$x$, $y$ to represent the same rational function $r$. This causes
a problem when we try to evaluate a rational function and to define
its domain over some algebra. For example, let us consider the evaluation
of the above rational function $r\left(x,y\right)$ on some unital
algebra $\mathcal{A}$. From the first representation $\left(xy\right)^{-1}$,
it gives a domain
\[
D_{1}=\left\{ \left(a,b\right)\in\mathcal{A}^{2}|ab\text{ is invertible in }\mathcal{A}\right\} ,
\]
on which the function $r$ is well-defined. But from the second one
$y^{-1}x^{-1}$, its domain is
\[
D_{2}=\left\{ \left(a,b\right)\in\mathcal{A}^{2}|a,b\text{ are invertible in }\mathcal{A}\right\} .
\]
Clearly $D_{2}\subseteq D_{1}$, but in general, we won't have $D_{1}\subseteq D_{2}$.
For example, if $\mathcal{A}=B\left(H\right)$ for some infinitely
dimensional Hilbert space, and $l$ is the one-sided left-shift operator,
then $l^{\ast}$ is the right-shift operator and we have the property
$l\cdot l^{\ast}=1$ but $l^{\ast}\cdot l\neq1$. Therefore, we see
that $\left(l,l^{\ast}\right)\notin D_{2}$ since both of them are
not invertible but $\left(l,l^{\ast}\right)\in D_{1}$.

Furthermore, if we want to evaluate a rational function $r$ which
has two different representations $\hat{r}_{1}$ and $\hat{r}_{2}$,
then we need to guarantee that for each element in the intersection
of the domains of $\hat{r}_{1}$ and $\hat{r}_{2}$, their evaluations
will agree. But this is also not true in general. To see this, we
can consider the following example,
\[
r\left(x,y\right)=1=y\left(xy\right)^{-1}x.
\]
Let $l$, $l^{*}$ be the left-shift and right-shift operators again,
then we see that $l^{\ast}(ll^{\ast})^{-1}l=l^{\ast}l\neq1$.

Thanks to the insights of Cohn, we can avoid such a problem by considering
an algebra $\mathcal{A}$ which is stably finite, i.e., for each $n\in M_{n}\left(\mathcal{A}\right)$,
any $A,B\in M_{n}\left(\mathcal{A}\right)$, we have that $AB=1$
implies $BA=1$. In fact, an algebra $\mathcal{A}$ is stable finite
if and only if all such representations of the zero function on the
algebra give zero evaluation. See Theorem 7.8.3 in the book \cite{Coh06}.
It is clear that $M_{n}\left(\mathbb{C}\right)$ is stably finite
for any $n\in\mathbb{N}$, so we can plug in our random matrices when
they are in the domain. And fortunately, any $\text{C}^{\ast}$-probability
space with a faithful trace is also stably finite (for a proof of
this fact, see Lemma 2.2 in \cite{HMS15}). So in this case, the evaluation
is well-defined if the elements are in the domain of the considered
representation.

In some sense, the above representations of rational functions are
the ``irreducible'' ones. That is, for a rational function $r\in\mathscr{R}$,
we can always take more times of algebraic operations than we really
need. For example, we can write
\[
\mathscr{R}_{0}\ni1=x^{-1}x=\left(x+yy^{-1}\right)^{-1}\left(x+zz^{-1}\right)=\cdots
\]
In order to obtain the maximal domain of a rational function, it's
much safer that we take the union of all the domains given by any
possible representations that can be ``reduced'' to the same rational
function.

Now we want to give a formal definition of such representations or
expressions, and show that they have a similar recursive structure
as rational functions $\mathscr{R}$. Then we can define the domains
of these rational expressions and hence the domains of rational functions.

Denoting $\mathfrak{R}_{0}=\mathscr{P}$, we define $\mathfrak{R}_{1}$
to the free complex algebra with generating set $\mathfrak{R}_{0}\cup\mathfrak{R}_{0}^{-1}$,
i.e., we view the polynomials and their inverses as letters instead
of rational functions in $\mathscr{R}$. In particular, $0^{-1}$
is also a valid non-empty word though it is meaningless when we try
to consider it as rational functions. Then we build the free algebra
with all words over this alphabet $\mathfrak{R}_{0}\cup\mathfrak{R}_{0}^{-1}$.
As a remark, we should note that for a polynomial, says $x$, the
words $x^{-1}\cdot x$, $x\cdot x^{-1}$ and $1$ are different words
in $\mathfrak{R}_{1}$, and $0$ is a non-empty word in $\mathfrak{R}_{1}$.

Therefore, we can construct a sequence of free algebra $\mathfrak{R}_{n}$,
$n\in\mathbb{N}$ recursively, that is, each $\mathfrak{R}_{n}$ is
just the free algebra generated by the alphabet $\mathfrak{R}_{n-1}\cup\mathfrak{R}_{n-1}^{-1}$,
$n\geqslant1$. It is clear that we have a natural inclusion map $i_{n}:\mathfrak{R}_{n}\rightarrow\mathfrak{R}_{n+1}$,
$n\in\mathbb{N}$ and hence we have their direct limit, denoted by
$\mathfrak{R}_{\infty}$.

Now we define $\phi_{0}:\mathfrak{R}_{0}\rightarrow\mathscr{R}_{0}$
as the identity map on polynomials. Then we can define a homomorphism
$\phi_{1}:\mathfrak{R}_{1}\rightarrow\mathscr{R}_{1}$ through extending
the map
\[
\phi_{1}\left(\alpha\right)=\begin{cases}
\phi_{0}\left(\alpha\right) & \alpha\text{ is a letter in the set }\mathfrak{R}_{0},\\
\left(\phi_{0}\left(\beta\right)\right)^{-1} & \alpha=\beta^{-1}\text{ is a letter in the set }\mathfrak{R}_{0}^{-1},\ \beta\not=0,\\
0 & \alpha=0^{-1}.
\end{cases}
\]
Therefore, we can define a sequence of homomorphisms $\left\{ \phi_{n}\right\} _{n\in\mathbb{N}}$
recursively, that is, by extending the map
\[
\phi_{n}\left(\alpha\right)=\begin{cases}
\phi_{n-1}\left(\alpha\right) & \alpha\text{ is a letter in }\mathfrak{R}_{n-1},\\
\left(\phi_{n-1}\left(\beta\right)\right)^{-1} & \alpha=\beta^{-1}\text{ is a letter in }\mathfrak{R}_{n-1}^{-1},\ \beta\not\in\ker\phi_{n-1},\\
0 & \alpha=\beta^{-1}\text{ is a letter in }\mathfrak{R}_{n-1}^{-1},\ \beta\in\ker\phi_{n-1}.
\end{cases}
\]
Thus, we see that there is a homomorphism $\Phi:\mathfrak{R}_{\infty}\rightarrow\mathscr{R}_{\infty}=\mathscr{R}$.
In other words, we have commutative diagrams as following: for every
$n\in\mathbb{N}$,

\begin{center}
\begin{tikzpicture}
  \node (P) {$\mathscr{P}$};
  \node (R1) [right of = P] {$\mathfrak{R}_{1}$};
  \node (R2) [right of = R1] {$\mathfrak{R}_{2}$};
  \node (Rdots) [right of = R2] {$\cdots$};
  \node (Rn) [right of = Rdots] {$\mathfrak{R}_{n}$};
  \node (R8) [right of = Rn] {$\mathfrak{R}_{\infty}$};
  \node (p) [below of = P] {$\mathscr{P}$};
  \node (r1) [right of = p] {$\mathscr{R}_{1}$};
  \node (r2) [right of = r1] {$\mathscr{R}_{2}$};
  \node (rdots) [right of = r2] {$\cdots$};
  \node (rn) [right of = rdots] {$\mathscr{R}_{n}$};
  \node (r8) [right of = rn] {$\mathscr{R}_{\infty}$};
  \draw[->] (P) to node []{} (R1);
  \draw[->] (R1) to node []{} (R2);
  \draw[->] (R2) to node []{} (Rdots);
  \draw[->] (Rdots) to node []{} (Rn);
  \draw[->] (Rn) to node []{} (R8);
  \draw[->] (p) to node []{} (r1);
  \draw[->] (r1) to node []{} (r2);
  \draw[->] (r2) to node []{} (rdots);
  \draw[->] (rdots) to node []{} (rn);
  \draw[->] (rn) to node []{} (r8);
  \draw[->] (P) to node [right]{$\phi_{0}$} (p);
  \draw[->] (R1) to node [right]{$\phi_{1}$} (r1);
  \draw[->] (R2) to node [right]{$\phi_{2}$} (r2);
  \draw[->] (Rn) to node [right]{$\phi_{n}$} (rn);
  \draw[->] (R8) to node [right]{$\Phi$} (r8);
\end{tikzpicture}
\par\end{center}

It is clear that $\Phi$ is surjective, so for a rational function
$r$ in $\mathscr{R}$, each element in its preimage $\Phi^{-1}\left(r\right)$
is a representation of $r$, and we call it a \emph{rational expression}
of $r$. As a word over some alphabet, the evaluation of a rational
expression at a tuple of elements in an algebra is clear, and thus
the \emph{domain} of a rational expression is the set of any tuple
that makes the evaluation possible. As mentioned previously, if an
algebra $\mathcal{A}$ is stably finite, then the evaluation of a
rational expression depends only on the corresponding rational function.
We define the \emph{domain} of a rational function $r$ as the union
of the domains of all possible rational expressions in $\Phi^{-1}\left(r\right)$. 

As a remark, we can see that the elements in $\ker\Phi$ arise as
the representations which can be ``reduced'' to $0$, such as
\[
y^{-1}\left(x^{-1}+y^{-1}\right)^{-1}x^{-1}-\left(x+y\right)^{-1},
\]
or which can be ``reduced'' to $0^{-1}$, like
\[
\left[1-y\left(xy\right)^{-1}x\right]^{-1},
\]
which make no sense when we evaluate them on algebras, and thus always
have the empty domain.

Now all the ingredients for rational functions are ready. But before
we move on to the convergence problem, we make just two more remarks
about rational expressions and functions, which may be helpful for
better understanding on this subject.

First, the rational expressions also give us another way to rediscover
the rational functions. Let $\mathcal{A}=\bigsqcup\limits _{n=1}^{\infty}\left[\left(M_{n}\left(\mathbb{C}\right)\right)^{r}\right]$
be the algebra consisting of all $r$-tuples of matrices of all sizes.
Then we can define that two rational expressions $\hat{r}_{1}$ and
$\hat{r}_{2}$ are ``equivalent'' if $\hat{r}_{1}\left(a\right)=\hat{r}_{2}\left(a\right)$
for each $a\in\text{dom}\left(\hat{r}_{1}\right)\cap\text{dom}\left(\hat{r}_{2}\right)\subseteq\mathcal{A}$.
Then it can be shown that these equivalence classes of rational expressions
coincide with the rational functions (for details, see \cite{KV12}).

At last, we want to emphasize again that these rational functions
or expressions are not just abstract objects from non-commutative
ring theory, but also appear in system and control theory, from the
theory of finite automata and formal languages to robust control and
linear matrix inequalities. In fact, they already use rational expressions
to consider related problems for about 50 years there. For example,
in the ``regular'' case, i.e., the rational expression with non
zero value at point $0$, the language of power series is applied
and first appeared in the theory of formal languages and finite automata
quite long ago, see Kleene \cite{Kle56}, Schützenberger \cite{Sch61,Sch63}
and Fliess \cite{Fli74a,Fli74b}. For a good exposition on this, see
the monograph by Berstel and Reutenauer \cite{BR84}.

\section{Convergence of the norm and trace for rational expressions}

Now we know enough to move on to our strong convergence problem of
rational functions. Equivalently, we will just consider rational expressions
due to our discussion in the last section. First of all, for a given
rational expression $r$ and a given tuple $x=(x_{1},\cdots,x_{m})$
in some $\text{C}^{\ast}$-probability space $(\mathcal{A},\tau)$
with faithful trace $\tau$, an assumption that $\left(x,x^{\ast}\right)$
lies in the domain of $r$ is reasonable. However, if there is a sequence
of tuples $x^{\left(n\right)}=(x_{1}^{\left(n\right)},\cdots,x_{m}^{\left(n\right)})$
from faithful tracial $\text{C}^{\ast}$-probability spaces $(\mathcal{A}^{\left(n\right)},\tau^{\left(n\right)})$
s.t. $x^{\left(n\right)}$ strongly converges to $x$, then it's not
necessary to assume that $(x^{\left(n\right)},(x^{\left(n\right)})^{\ast})$
also lies the domain of $r$. It turns out that we can deduce this
well-definedness of $r(x^{\left(n\right)},(x^{\left(n\right)})^{\ast})$
for sufficiently large $n$.
\begin{thm}
Suppose that $x^{\left(n\right)}$ strongly converges to $x$ and
the tuple $(x,x^{\ast})$ lies in the domain of a rational function
$r\in\mathscr{R}$. Then we have
\begin{enumerate}
\item $(x^{\left(n\right)},(x^{\left(n\right)})^{\ast})$ lies in the domain
of $r$ eventually;
\item the convergence of norms, i.e.,
\[
\lim_{n\rightarrow\infty}\left\Vert r(x^{\left(n\right)},(x^{\left(n\right)})^{\ast})\right\Vert _{_{\mathcal{A}^{\left(n\right)}}}=\left\Vert r(x,x^{\ast})\right\Vert _{_{\mathcal{A}}}.
\]
\end{enumerate}
\end{thm}
\begin{proof}
We will prove our main theorem in a recursive way based on the description
of rational expressions in the last section. That is, we want to prove
the above statement by induction on $\mathfrak{R}_{k}$, $k=0,1,2,\cdots$.
For $k=0$, it is the convergence for polynomials, which is just our
assumption. Thus, we suppose that the above two statements hold for
any rational expression $r\in\mathfrak{R}_{k}$ and we are going to
prove them for $\mathfrak{R}_{k+1}$.

First, we need to check the domain problem. Since each rational expression
in $\mathfrak{R}_{k+1}$ can be represented as a finite sum of products
of some rational expressions in $\mathfrak{R}_{k}$ and their inverses,
we only need to prove that, for any $\hat{r}\in\mathfrak{R}_{k}$
with $(x,x^{\ast})$ in the domain of $\hat{r}^{-1}\in\mathfrak{R}_{k+1}$,
$(x^{\left(n\right)},(x^{\left(n\right)})^{\ast})$ lies in the domain
of $\hat{r}^{-1}$ eventually. Or in other words, if $\hat{r}\left(x,x^{\ast}\right)$
is invertible as an operator in $\mathcal{A}$, then $\hat{r}(x^{\left(n\right)},(x^{\left(n\right)})^{\ast})$
is invertible in $\mathcal{A}^{\left(n\right)}$ for sufficiently
large $n$.

For a rational expression, say $\hat{r}$, we always denote $\hat{r}^{\left(\infty\right)}=\hat{r}\left(x,x^{\ast}\right)$,
$\hat{r}^{\left(n\right)}=\hat{r}(x^{\left(n\right)},(x^{\left(n\right)})^{\ast})$.
Because $\hat{r}^{\left(\infty\right)}(\hat{r}^{\left(\infty\right)})^{\ast}$
is positive and invertible, we have
\[
\left\Vert R^{\left(\infty\right)}-\hat{r}^{\left(\infty\right)}(\hat{r}^{\left(\infty\right)})^{\ast}\right\Vert <R^{\left(\infty\right)}
\]
where $R^{\left(\infty\right)}=\parallel\hat{r}^{\left(\infty\right)}(\hat{r}^{\left(\infty\right)})^{\ast}\parallel>0$.
By the assumption, we know
\begin{align}
\left\Vert R^{\left(\infty\right)}-\hat{r}^{\left(\infty\right)}(\hat{r}^{\left(\infty\right)})^{\ast}\right\Vert  & =\left\Vert R^{\left(\infty\right)}-(\hat{r}(\hat{r})^{\ast})^{\left(\infty\right)}\right\Vert \nonumber \\
 & =\lim_{n\rightarrow\infty}\left\Vert R^{\left(\infty\right)}-(\hat{r}(\hat{r})^{\ast})^{\left(n\right)}\right\Vert \label{eq:*}
\end{align}
because $R^{\left(\infty\right)}-\hat{r}(\hat{r})^{\ast}$ is a rational
expression in $\mathfrak{R}_{k}$. Then, denoting $R^{\left(n\right)}=\left\Vert \hat{r}^{\left(n\right)}(\hat{r}^{\left(n\right)})^{\ast}\right\Vert $,
from the inverse triangle inequality 
\[
\left|\left\Vert R^{\left(n\right)}-\hat{r}^{\left(n\right)}(\hat{r}^{\left(n\right)})^{\ast}\right\Vert -\left\Vert R^{\left(\infty\right)}-\hat{r}^{\left(n\right)}(\hat{r}^{\left(n\right)})^{\ast}\right\Vert \right|\leqslant\left|R^{\left(n\right)}-R^{\left(\infty\right)}\right|
\]
and 
\[
R^{\left(\infty\right)}=\lim\limits _{n\rightarrow\infty}R^{\left(n\right)},
\]
it follows that
\[
\lim_{n\rightarrow\infty}\left\Vert R^{\left(n\right)}-\hat{r}^{\left(n\right)}(\hat{r}^{\left(n\right)})^{\ast}\right\Vert =\lim_{n\rightarrow\infty}\left\Vert R^{\left(\infty\right)}-\hat{r}^{\left(n\right)}(\hat{r}^{\left(n\right)})^{\ast}\right\Vert .
\]
Hence, combining with (\ref{eq:*}), we have
\begin{eqnarray*}
\lim_{n\rightarrow\infty}\left(R^{\left(n\right)}-\left\Vert R^{\left(n\right)}-\hat{r}^{\left(n\right)}(\hat{r}^{\left(n\right)})^{\ast}\right\Vert \right) & = & R^{\left(\infty\right)}-\lim_{n\rightarrow\infty}\left\Vert R^{\left(n\right)}-\hat{r}^{\left(n\right)}(\hat{r}^{\left(n\right)})^{\ast}\right\Vert \\
 & = & R^{\left(\infty\right)}-\lim_{n\rightarrow\infty}\left\Vert R^{\left(\infty\right)}-\hat{r}^{\left(n\right)}(\hat{r}^{\left(n\right)})^{\ast}\right\Vert \\
 & = & R^{\left(\infty\right)}-\left\Vert R^{\left(\infty\right)}-\hat{r}^{\left(\infty\right)}(\hat{r}^{\left(\infty\right)})^{\ast}\right\Vert \\
 & > & 0.
\end{eqnarray*}
This implies that
\[
\left\Vert R^{\left(n\right)}-\hat{r}^{\left(n\right)}(\hat{r}^{\left(n\right)})^{\ast}\right\Vert <R^{\left(n\right)}
\]
for $n$ large enough, which is equivalent to say $\hat{r}^{\left(n\right)}(\hat{r}^{\left(n\right)})^{\ast}$
is invertible eventually. Recall that $\left(\mathcal{A}^{\left(n\right)},\tau^{\left(n\right)}\right)$
is stable finite, so we can easily deduce that $\hat{r}^{\left(n\right)}$
is also invertible because it has a right inverse $(\hat{r}^{\left(n\right)})^{\ast}(\hat{r}^{\left(n\right)}(\hat{r}^{\left(n\right)})^{\ast})^{-1}$
when $n$ is large enough.

Moreover, denoting by $\sigma\left(a\right)$ the spectrum of an operator
$a$, we can see that
\begin{eqnarray}
\left\Vert (\hat{r}^{\left(\infty\right)})^{-1}\right\Vert  & = & \sqrt{\left\Vert (\hat{r}^{\left(\infty\right)}(\hat{r}^{\left(\infty\right)})^{\ast})^{-1}\right\Vert }\nonumber \\
 & = & \sqrt{\left(\min\sigma(\hat{r}^{\left(\infty\right)}(\hat{r}^{\left(\infty\right)})^{\ast})\right)^{-1}}\nonumber \\
 & = & \sqrt{\left(R^{\left(\infty\right)}-\left\Vert R^{\left(\infty\right)}-\hat{r}^{\left(\infty\right)}(\hat{r}^{\left(\infty\right)})^{\ast}\right\Vert \right)^{-1}}\nonumber \\
 & = & \lim_{n\rightarrow\infty}\sqrt{\left(R^{\left(n\right)}-\left\Vert R^{\left(n\right)}-\hat{r}^{\left(n\right)}(\hat{r}^{\left(n\right)})^{\ast}\right\Vert \right)^{-1}}\nonumber \\
 & = & \lim_{n\rightarrow\infty}\sqrt{\left(\min\sigma(\hat{r}^{\left(n\right)}(\hat{r}^{\left(n\right)})^{\ast})\right)^{-1}}\nonumber \\
 & = & \lim_{n\rightarrow\infty}\sqrt{\left\Vert \left(\hat{r}^{\left(n\right)}(\hat{r}^{\left(n\right)})^{\ast}\right)^{-1}\right\Vert }\nonumber \\
 & = & \lim_{n\rightarrow\infty}\left\Vert (\hat{r}^{\left(n\right)})^{-1}\right\Vert .\label{eq:**}
\end{eqnarray}

Now, considering a rational expression $\hat{r}\in\mathfrak{R}_{k+1}$
s.t. its domain contains $\left(x,x^{\ast}\right)$, then, by the
above argument, we can see that $(x^{\left(n\right)},(x^{\left(n\right)})^{\ast})$
lies in the domain of $\hat{r}$ eventually. That is, there is $N\in\mathbb{N}$
s.t. $(x^{\left(n\right)},(x^{\left(n\right)})^{\ast})$ is in the
domain of $\hat{r}$ for all $n>N$. Setting 
\[
\mathcal{M}=\left\{ \left(a,a^{\left(N+1\right)},a^{\left(N+2\right)},\cdots\right)\in\mathcal{A}\times\prod_{n>N}\mathcal{A}^{\left(n\right)}|\max\left\{ \left\Vert a\right\Vert ,\sup_{n>N}\left\Vert a^{\left(n\right)}\right\Vert \right\} <\infty\right\} ,
\]
then $\mathcal{M}$ is $\text{C}^{*}$-algebra with the norm
\[
\left\Vert \left(a,a^{\left(N+1\right)},\cdots\right)\right\Vert =\sup\left\{ \left\Vert a\right\Vert ,\sup_{n>N}\left\Vert a^{\left(n\right)}\right\Vert \right\} .
\]
We put
\[
X_{i}=\left(x_{i},x_{i}^{\left(N+1\right)},\cdots\right),
\]
for $1\leqslant i\leqslant m$, then $X_{i}\in\mathcal{M}$. Moreover,
denoting $X=(X_{1},\cdots,X_{m})$, we have $\left(X,X^{\ast}\right)$
lies in the domain of $\hat{r}$ over $\mathcal{M}^{2m}$, namely,
\[
\hat{r}\left(X,X^{\ast}\right)=\left(\hat{r}^{\left(\infty\right)},\hat{r}^{\left(N+1\right)},\cdots\right)
\]
is well defined. Furthermore, we can see $\hat{r}\left(X,X^{\ast}\right)$
is also in $\mathcal{M}$. In fact, recall that $\hat{r}$ can be
written as a finite sum of products consisting of rational expressions
in $\mathfrak{R}_{k}$ and of their inverses, which are all bounded
because of (\ref{eq:**}), i.e., for each $\hat{s}\in\mathfrak{R}_{k}$,
$\left\Vert (\hat{s}^{\left(\infty\right)})^{-1}\right\Vert =\lim\limits _{n\rightarrow\infty}\left\Vert (\hat{s}^{\left(n\right)})^{-1}\right\Vert <\infty$.
It follows that
\[
\max\left\{ \left\Vert \hat{r}^{\left(\infty\right)}\right\Vert ,\sup_{n>N}\left\Vert \hat{r}^{\left(n\right)}\right\Vert \right\} <\infty,
\]
which means that $\hat{r}\left(X,X^{\ast}\right)\in\mathcal{M}$.

Therefore, $\hat{r}\left(X,X^{\ast}\right)$ lies in the sub $\text{C}^{\ast}$-algebra
of $\mathcal{M}$ generated by $\left(X,X^{\ast}\right)$ because
an invertible element is still invertible in any sub $\text{C}^{\ast}$-algebra
containing it (see Proposition 4.1.5 in \cite{KR83}). Thus, for any
$\varepsilon>0$, we can find a polynomial $p$ s.t.
\begin{equation}
\left\Vert p\left(X,X^{\ast}\right)-\hat{r}\left(X,X^{\ast}\right)\right\Vert <\varepsilon.\label{eq: approximating polynomial}
\end{equation}
In particular, we have
\[
\left\Vert p^{\left(\infty\right)}-\hat{r}^{\left(\infty\right)}\right\Vert <\varepsilon
\]
and
\[
\left\Vert p^{\left(n\right)}-\hat{r}^{\left(n\right)}\right\Vert <\varepsilon
\]
for all $n>N$. Hence,
\begin{eqnarray*}
\left|\left\Vert \hat{r}^{\left(n\right)}\right\Vert -\left\Vert \hat{r}^{\left(\infty\right)}\right\Vert \right| & \leqslant & \left|\left\Vert \hat{r}^{\left(n\right)}\right\Vert -\left\Vert p^{\left(n\right)}\right\Vert \right|+\left|\left\Vert p^{\left(n\right)}\right\Vert -\left\Vert p^{\left(\infty\right)}\right\Vert \right|+\left|\left\Vert p^{\left(\infty\right)}\right\Vert -\left\Vert \hat{r}^{\left(\infty\right)}\right\Vert \right|\\
 & \leqslant & \left\Vert \hat{r}^{\left(n\right)}-p^{\left(n\right)}\right\Vert +\left|\left\Vert p^{\left(n\right)}\right\Vert -\left\Vert p^{\left(\infty\right)}\right\Vert \right|+\left\Vert p^{\left(\infty\right)}-\hat{r}^{\left(\infty\right)}\right\Vert \\
 & \leqslant & 2\varepsilon+\left|\left\Vert p^{\left(n\right)}\right\Vert -\left\Vert p^{\left(\infty\right)}\right\Vert \right|
\end{eqnarray*}
for any $n>N$. Combining this with the fact that
\[
\lim_{n\rightarrow\infty}\left\Vert p^{\left(n\right)}\right\Vert =\left\Vert p^{\left(\infty\right)}\right\Vert ,
\]
we have
\[
\limsup_{n\rightarrow\infty}\left|\left\Vert \hat{r}^{\left(n\right)}\right\Vert -\left\Vert \hat{r}^{\left(\infty\right)}\right\Vert \right|<2\varepsilon.
\]
Since $\varepsilon$ is arbitrary, we obtain the result of convergence
of norm.
\end{proof}
An immediate consequence of the theorem is that we also have the convergence
in trace for rational functions.
\begin{cor}
Suppose that $x^{\left(n\right)}$ strongly converges to $x$ and
the tuple $\left(x,x^{\ast}\right)$ lies in the domain of a rational
function $r\in\mathscr{R}$, then we have
\[
\lim_{n\rightarrow\infty}\tau^{\left(n\right)}\left(r(x^{\left(n\right)},(x^{\left(n\right)})^{\ast})\right)=\tau\left(r(x,x^{\ast})\right).
\]
\end{cor}
\begin{proof}
We can see that a similar argument as in the proof of previous theorem
also works for the convergence in trace. Assume that $\hat{r}$ is
a rational expression, $\left(X,X^{\ast}\right)$ and polynomial $p$
are as above s.t. (\ref{eq: approximating polynomial}) holds. Then
\begin{eqnarray*}
 &  & \left|\tau^{\left(n\right)}(\hat{r}^{\left(n\right)})-\tau(\hat{r}^{\left(\infty\right)})\right|\\
 & \leqslant & \left|\tau^{\left(n\right)}(\hat{r}^{\left(n\right)}-p^{\left(n\right)})\right|+\left|\tau^{\left(n\right)}(p^{\left(n\right)})-\tau(p^{\left(\infty\right)})\right|+\tau(p^{\left(\infty\right)}-\hat{r}^{\left(n\right)})\\
 & \leqslant & \left\Vert \hat{r}^{\left(n\right)}-p^{\left(n\right)}\right\Vert +\left|\tau^{\left(n\right)}(p^{\left(n\right)})-\tau(p^{\left(\infty\right)})\right|+\left\Vert \hat{r}^{\left(\infty\right)}-p^{\left(\infty\right)}\right\Vert \\
 & \leqslant & 2\varepsilon+\left|\tau^{\left(n\right)}(p^{\left(n\right)})-\tau(p^{\left(\infty\right)})\right|
\end{eqnarray*}
for $n$ large enough. From the fact that $\lim\limits _{n\rightarrow\infty}\tau^{\left(n\right)}(p^{\left(n\right)})=\tau(p^{\left(\infty\right)})$,
it follows, by letting $\varepsilon$ tend to 0, that
\[
\lim_{n\rightarrow\infty}\tau^{\left(n\right)}(\hat{r}^{\left(n\right)})=\tau(\hat{r}^{\left(\infty\right)}).
\]
\end{proof}
Finally, we give two remarks on possible further investigations.

First, as mentioned in the Introduction, the strong convergence is
stable under taking reduced free products (\cite{Sko15} and \cite{Pis16}),
that is, if $x^{\left(n\right)}$ and $y^{\left(n\right)}$ are $\ast$-free
for each $n\in\mathbb{N}$ and have strong limits $x$ and $y$ respectively,
then $\left(x,y\right)$ is the strong limit of $(x^{\left(n\right)},y^{\left(n\right)})$.
The analogue for weak convergence is also true, that is, the convergence
in distribution is also stable under the reduced free product, namely,
if $x^{\left(n\right)}$ and $y^{\left(n\right)}$ are $\ast$-free
and have $x$ and $y$ as their limits in distribution respectively,
then $\left(x,y\right)$ is the limit of $\left(x^{\left(n\right)},y^{\left(n\right)}\right)$
in distribution. Some similar results for strongly convergent random
matrices are mentioned in the Introduction (\cite{Mal12}, \cite{CM14}
and \cite{BC16}), where we can adjoin two asymptotic free tuples
of random matrices. And the analogue for convergence in distribution,
also holds under certain conditions for random matrices (\cite{HP00}).

Therefore, as we have seen that strong convergence is stable under
taking inverses, it is natural to ask if convergence in distribution
is also stable under taking inverses. So assume that $x^{\left(n\right)}=(x_{1}^{\left(n\right)},\cdots,x_{m}^{\left(n\right)})$
converges in distribution to $x=(x_{1},\cdots,x_{m})$, i.e.,
\[
\lim_{n\rightarrow\infty}\tau^{\left(n\right)}\left(p(x^{\left(n\right)},(x^{\left(n\right)})^{\ast})\right)=\tau\left(p\left(x,x^{\ast}\right)\right)
\]
for any polynomial $p$, the question is whether we can from this
conclude that
\[
\lim_{n\rightarrow\infty}\tau^{\left(n\right)}\left(r(x^{\left(n\right)},(x^{\left(n\right)})^{\ast})\right)=\tau\left(r\left(x,x^{\ast}\right)\right)
\]
for a rational function $r$, under certain assumptions but without
assuming strong convergence. To consider this convergence for random
matrices does make sense because it is well known that some random
matrices converge in distribution but not strongly. For example, a
Wigner matrix $A=(a_{ij})_{i,j=1}^{n}$ whose entries are uniformly
bounded i.i.d. random variable s.t. $\mathbb{E}(a_{11})=\mu>0$, has
its largest eigenvalue asymptotically outside the support of the semi-circular
law (for a reference, see \cite{FK81}).

Unfortunately, it seems that outliers make the convergence in distribution
unstable with respect to inverses. Here is a simple example: let $X^{\left(n\right)}\in M_{n}\left(\mathbb{C}\right)$
be a sequence of matrices that strongly converges to $x$, which lies
in some faithful tracial $\text{C}^{\ast}$-probability space $\left(\mathcal{A},\tau\right)$.
We assume that $x$ is invertible, then by our main theorem, we have
$X^{\left(n\right)}$ is invertible eventually, and
\[
\lim_{n\rightarrow\infty}\text{tr}_{n}\left((X^{\left(n\right)})^{-1}\right)=\tau\left(x^{-1}\right).
\]
Now put
\[
Y^{\left(n+1\right)}=\begin{pmatrix}\frac{1}{n+1} & 0\\
0 & X^{\left(n\right)}
\end{pmatrix}\in M_{n+1}\left(\mathbb{C}\right),
\]
then it is clear that $Y^{\left(n\right)}$ also converges in distribution
to $x$ and $Y^{\left(n\right)}$ is invertible as $X^{\left(n\right)}$
is invertible eventually. However, we can see that 
\[
\lim_{n\rightarrow\infty}\text{tr}_{n}\left((Y^{\left(n\right)})^{-1}\right)=1+\tau\left(x^{-1}\right).
\]

Secondly, if we consider in the one-variable case, a sequence of self-adjoint
random variables $\{x^{\left(n\right)}\}_{n\geqslant1}$ which strongly
converges to a self-adjoint random variable $x$, then for any continuous
function $f$ defined on a neighborhood of the interval $\left[-\left\Vert x\right\Vert ,\left\Vert x\right\Vert \right]$,
we can see that $f(x^{\left(n\right)})$ will be eventually well-defined
since the support of $x^{\left(n\right)}$ is approaching to $\left[-\left\Vert x\right\Vert ,\left\Vert x\right\Vert \right]$.
On the other hand, since we can find some polynomials $\{p_{k}\}$
uniformly converging to $f$ on this neighborhood, we can use the
same argument as above to show that $f(x^{\left(n\right)})$ converges
to $f\left(x\right)$ in trace and in norm. However, for the general
multivariable case, it is not clear whether one can go beyond the
case of rational functions. Nevertheless, it is tempting to hope to
be able to extend our investigation to the case of non-commutative
analytic functions.

\end{document}